\documentclass{amsart}
\usepackage{amssymb}
\usepackage{latexsym}

\input xy
\xyoption{all}
\newdir{ >}{{}*!/-12pt/@{>}}

\newcommand{\cof}{\rightarrowtail}

\newcommand{\cab}[2]{
        \begin{picture}(20,30)(0,0)
        \put(-4,-5){\vector(0,-1){#1}}
        \put(-4,-3){\makebox(0,0){$\vee$}}
        \put(5,-13){\makebox(0,0){#2}}
        \end{picture}     }

\newtheorem{thm}{Theorem}[section]
\newtheorem{prop}[thm]{Proposition}
\newtheorem{cor}[thm]{Corollary}
\newtheorem{lem}[thm]{Lemma}
\newtheorem{dfn}[thm]{Definition}

\newtheorem{remark}[thm]{Remark}

\begin{document}

%\begin{frontmatter}

\title{Generalized Homotopy theory in Categories with a Natural
Cone}

\author{Francisco J. D\'{\i}az and Jos\'{e} M. G. Calcines }

\address{${}$ \newline Departamento de Matem\'{a}tica Fundamental, Universidad de La Laguna. \newline 38271 La Laguna.}

\thanks{This work has been supported by the Ministerio de Educaci\'{o}n y Ciencia grant MTM2009-12081 and FEDER}

\keywords{Category, algebraic homotopy theory, cone construction, generalized homotopy.}

\subjclass[2000]{55U35, 18C, 55P05.}

\begin{abstract}
In proper homotopy theory, the original concept of point used in
the classical homotopy theory of topological spaces is generalized
in order to obtain homotopy groups that study the infinite of the
spaces. This idea: ``Using any arbitrary object as base point''
and even ``any morphism as zero morphism'' can be developed in
most of the algebraic homotopy theories. In particular, categories with a natural cone have a generalized
homotopy theory obtained through the relative homotopy relation.
Generalized homotopy groups and exact sequences of them are built so
that respective classical pointed ones are a particular case of these.
\end{abstract}

\maketitle

%\end{frontmatter}

%{\footnotesize \it AMS Subject Class. (2000): Primary: .}

\section{Introduction}
The main problem to build a homotopy theory based on the concept
of cone object is to define the homotopy relation between
morphisms through the notion of nullhomotopic morphism. In
additive categories this problem can be solved using the addition
of morphisms in the category, and generalizing the injective
homotopy theory defined by Eckmann and Hilton \cite{Hi} for
$R$-modules. Algebraic theories in this sense were created by H.
Kleisli \cite{Kl} and J. A. Seebach \cite{S}. Also, S. Rodr\'{\i}guez-Mach\'{\i}n gives an algebraic
homotopy theory based on a cone functor in additive categories,
without using injective objects, that contains the injective
homotopy theories mentioned above as particular cases \cite{P-}. A
generalization to arbitrary categories has been given by S. Rodr\'{\i}guez-Mach\'{\i}n and the first 
author of this paper in \cite{D-}: Identifying the notion of dual
standard construction given by P. J. Huber \cite{Hu} with the
concept of cone, and adapting the axioms given by H. J. Baues
about cofibrations in Categories with a Natural Cylinder \cite{B}
to a cone object obtained by collapsing the base of the cylinder
to a single point. The algebraic homotopy theory in this way
defined attains that classical homotopy theory of the topological
spaces and pointed topological spaces can be developed using their
respective cones. Also the proper homotopy theory of the
topological spaces can be developed in this sense, through a cone
functor.

Finally, in categories with a natural cylinder in the sense of
Baues, the cone functor obtained by collapsing, using pushout
diagrams, the base of the cylinder functor to a single point gives
the same homotopy theory that the cylinder functor.

In order to study the infinite of topological spaces, the proper
homotopy theory defines several homotopy groups using different
spaces as base point. So, the Brown homotopy groups \cite{Br} can be
defined using a sequence of points a base point, and the Steenrod
homotopy groups \cite{C} use spaces based on a {\it base ray}. H.J. Baues
uses trees as base point of the spaces to create a proper homotopy
theory \cite{B2} \cite{B3}. These facts suggest a more general
concept of base point in homotopy theory. Moreover, in many theories
homotopy groups and exact sequences of them can be defined using
arbitrary objects and morphisms as base point and zero morphism,
respectively. These homotopy theories are called {\it generalized.}

A generalized homotopy theory can be defined in categories with a
natural cone using relative homotopy: Generalized homotopy groups are
homotopy groups relative to a cofibration based on a morphism. Also,
generalized exact homotopy sequences of these groups are given. In
this way, when the category is pointed, the classical homotopy theory
is a generalized homotopy theory based on a point.

Finally, fixed an object as base point in a category with a
natural cone, a generalization of the process used to obtain
spheres beginning with a point in topological spaces allows one to
define spheres beginning with the fixed object.

We point out that the main results of this paper were already announced 
in \cite{D-1}.

\section{Notation and preliminaries}
The following categorical notation will be useful in the
interpretation of this paper.

Given functors ${\bf B}\stackrel{E}{\to} {\bf
C}\stackrel{F,G}{\longrightarrow} {\bf D}\stackrel{H}{\to}{\bf E}$
and a transformation $t:F\to G$, then the transformations
$t*E:FE\to GE$ and $H*t:HF\to HG$ will be denoted by $t_E$ and
$Ht$, respectively. When there is not a possibility of confusion,
the morphism $t_X:FX\to GX$ will be simply denoted by $t$, for
every object $X$.

The pushout object of two morphisms $f$ and $g$ will be denoted by
$P\{ f,g\}$. The induced morphisms will be denoted by
$\overline{f}:codom\;g \to P\{ f,g\}$ and
$\overline{g}:codom\;f\to P\{ f,g\}$. Given a morphism $f$, if the
notation $\overline{f}$ has been used, $\widetilde{f}$ will denote
other induced morphism by $f$ in a pushout. In particular, if
$f=g$ then $\overline{f}$ and $\widetilde{f}$ will denote the
morphisms $\overline{f}$ and $\overline{g}$, respectively.

Given morphisms $r$ and $s$ verifying $rf=sg$, the unique morphism
$h$ such that $h\overline{g}=r$ and $h\overline{f}=s$ will be
denoted by $\{ r,s\}$. If $codom \; f$ (resp. $codom \; g$) is a
pushout object, the component $r$ (resp. $s$) has an expression
like $\{ r_1,r_2\}$ (resp. $\{ s_1,s_2\}$). Frequently, in this
case, the morphism $\{ r,s\} = \{\{ r_1,r_2\} ,s\}$ (resp. $\{
r,\{ s_1,s_2\}\}$) will be denoted by $\{ r_1,r_2,s\}$ (resp. $\{
r,s_1,s_2\}$). In this way, expressions of the type $\{
h_0,h_1,...,h_n\}$ can appear.

Given two pushout objects $P\{ f,g\}$ and $P\{ f',g'\}$, and three
morphisms $r:codom\hspace{4pt}f\to codom\hspace{4pt}f'$, $s:codom
\hspace{4pt}g\to codom\hspace{4pt}g'$ and $t:dom\hspace{4pt}f=
dom\hspace{4pt}g\to dom\hspace{4pt}f'=dom\hspace{4pt}g'$ verifying
$rf=f't$ and $sg=g't$, we will denote the unique morphism $\{
\overline{g'}r, \overline{f'}s\}$ by $r\cup s$. If there is not
possibility of confusion, expressions of the type $h_0\cup
h_1\cup... \cup h_n$ will be used.

Finally, the set of {\it extensions} of a morphism $u:B\to X$
relative to other morphism $i:B\to A$ is defined by
$Hom(A,X)^{u(i)}=\{ f:A\to X \; / \; fi=u\}$

Next we recall some concepts given in \cite{D-} relative to a category with a natural cone.

\begin{dfn}\label{ccat}
A category with a natural cone, or $C$-category, is a category
{\bf C} together with  a class ``cof'' of morphisms in {\bf C},
called cofibrations and denoted by $\cof$, a functor $C:{\bf C}\to
{\bf C}$ which will be called the cone functor, and natural
transformations $\kappa:1\to C$ and $\rho:CC\to C$, denominated
inclusion and projection respectively, satisfying the following
axioms:
\begin{enumerate}
 \item[{\small {\bf C1.}}] {\bf Cone axiom.} $\rho \kappa_C = \rho C\kappa = 1_C$ and
$\rho\rho_C=\rho C\rho$.

 \item[{\small {\bf C2.}}] {\bf Pushout axiom.}
     For any pair of morphisms $A\stackrel {i}{\leftarrowtail} B\stackrel{f}{\to}
     X$, where $i$ is a cofibration, there exists the pushout square
     $$
     \xymatrix{
     {B} \ar[rr]^{f} \ar@{ >->}[d]_{i} & & {X} \ar@{ >->}[d]^{\overline{i}}  \\
     {A} \ar[rr]_{\overline{f}} & & {P\{ i,f\}}}
     $$
     and $\overline{i}$ is also a cofibration.
     The cone functor carries this pushout diagram (called cofibrated pushout) into a
     pushout diagram, that is $C(P\{ i,f\}) = P\{ Ci,Cf\}$.

 \item[{\small {\bf C3.}}] {\bf Cofibration axiom.} For each object $X$ the morphisms
     $1_X$ and $\kappa_X$ are cofibrations. The composition of
     two cofibrations is a cofibration. Moreover, there is a retraction for the cone of
     each cofibration. This last property is called
     nullhomotopy extension property (NEP).

 \item[{\small {\bf C4.}}] {\bf Relative cone axiom.} Given a cofibration $i:B\cof A$,
       the morphism $i_1=\{ Ci,\kappa \}:\Sigma^i=P\{ \kappa,i\}\cof CA$
       is also a cofibration. The object $\Sigma^i$ is called relative cone of
       $i$.
 \end{enumerate}
\end{dfn}

Note that isomorphisms and cones of cofibrations are also cofibrations.

Given a cofibration $i$, for each non-negative integer $n$ one
defines $i_n=(i_{n-1})_1$, with $i_0 =i$.

\begin{thm}\label{cubo}
Given the commutative cubical diagram $$\xymatrix{
 & {X} \ar[rr]^{g} \ar@{ >->}[dd]^<(.2){\gamma}
  \ar[dl]_{f} & & {Z} \ar@{ >->}[dd]^{\beta} \ar[dl]_{\overline{f}} \\
 {\mbox{\small Y}} \ar[rr]_{\overline{g}} \ar@{ >->}[dd]_{\alpha}
 & & {P\{ f,g\}} \ar@{ >->}[dd]^<(.2){\alpha\cup\beta} &                 \\
 & {X'} \ar[rr]^<(.2){g'} \ar[dl]_{f'} & & {Z'} \ar[dl]_{\overline{f'}}  \\
 {Y'} \ar[rr]_{\overline{g'}} & & {P\{f',g'\}} & }
$$ where the top and bottom faces are pushouts and
$\alpha, \beta, \gamma$ are cofibrations. If $\{g',\beta \}:{\rm
P}\{\gamma ,g\}\to {\rm Z'}$ or $\{f',\alpha\}:{\rm P}\{\gamma
,f\}\to {\rm Y'}$ is a cofibration then so is $\alpha\cup\beta$.
\end{thm}

\begin{dfn}
A morphism $f:X\to Y$ is said to be nullhomotopic ($f\simeq 0$) if
there exists an extension $F:CX\to Y$ of the morphism $f$ relative
to the cofibration $\kappa_X$. $F$ is called a nullhomotopy for
$f$ $(F:f\simeq 0)$.

An object X is said to be contractible ($X\simeq 0)$ when $1_{\rm
X}\simeq 0$.
\end{dfn}

So a morphism is nullhomotopic if and only if it may be factored
through a contractible object. Hence the nullhomotopy extension
property can be also stated in terms of nullhomotopy:

\begin{thm}[NEP]
Given a morphism $i:{\rm B}\to {\rm A}$, the following sentences
are equivalent:
\begin{enumerate}
\item[a)] The morphism $i$ verifies NEP.
\item[b)] Every nullhomotopic morphism $f:B\to X$ has a nullhomotopic
          extension rel. $i$.
\item[c)] Every nullhomotopic morphism $f:B\to X$ has an extension
          rel. $i$.
\item[d)] The inclusion $\kappa :B\to CB$ has an extension rel. $i$.
\end{enumerate}
\end{thm}

\begin{dfn}
A cofibration $i:B\cof A$ is said to be contractible when $B$ and $A$ are
contractible objects. These cofibrations will be the contractible
objects in the category of pairs.
\end{dfn}

Observe that pushout objects of two contractible cofibrations are
contractible objects. Hence if $i$ is contractible then $i_n$ and
$\Sigma^{i_{n-1}}$ so are, for each natural number $n$.

\begin{dfn}
Given a cofibration $i:B\cof CA$ and two morphisms
$f_{0},f_{1}:CA\to X$, $f_{0}$ is said to be homotopic to $f_{1}$
relative to the cofibration $i$ ($f_0\simeq f_1$ rel. $i$) if
there exists an extension $F:C^2A\to X$ of the morphism $\{
f_{0}\rho Ci,f_{1}\}$ relative to the cofibration $i_1$. $F$ is
called homotopy from $f_{0}$ to $f_{1}$ relative to $i$, in
symbols $F:f_{0}\simeq f_{1}$ rel. $i$.
\end{dfn}

\begin{remark}\label{hom}
The homotopy relation relative to a cofibration $i$ is an
equivalence relation compatible with the composition of morphisms
in the following sense:
\begin{enumerate}
\item[-] If $F:f_0\simeq f_1$ rel. $i$ then $fF:ff_0\simeq ff_1$ rel.
         $i$.
\item[-] Given a commutative square $(Cf)i=jg$, if $F:f_0\simeq
         f_1$ rel. $j$ then $FC^{2}f:f_{0}Cf\simeq f_{1}Cf$ rel. $i$.
         If the commutative square is a pushout then $f_0\simeq
         f_1$ rel. $j$ if and only if $f_{0}Cf\simeq f_{1}Cf$
         rel. $i$.
\end{enumerate}
\end{remark}

The quotient set $Hom(CA,X)^{u(i)}/\hspace{-4pt}\simeq$ will be denoted by
$[CA,X]^{u(i)}$, where $u=f_0i=f_1i$.

The following property is fundamental to obtain to obtain equalities
among morphisms save homotopy.

\begin{thm}\label{cont}
If $X$ or $i$ is contractible then $f_0\simeq f_1$ rel. $i$ if and
only if $f_0i=f_1i$.
\end{thm}

\begin{dfn}
A $C$-category is said to be pointed if every object $X$ is
cofibrant (that is, the initial morphism $\emptyset_X:\emptyset\to
X$ is a cofibration) and $C\emptyset = \emptyset$. In pointed
categories the initial object is denoted by $*$ and it is called
point.
\end{dfn}

If $X$ and $Y$ are objects of a pointed $C$-category, then $X\vee
Y$ will denote the pushout object $P\{*_X, *_Y\}$. Note that
$C(X\vee Y) = CX\vee CY$. If $X\simeq 0$ and $Y\simeq 0$ then
$X\vee Y\simeq 0$. If $i:B\cof A$ and $i':B'\cof A'$ are
cofibrations then $i\vee i':B\vee B'\to A\vee A'$ so is.

\begin{remark}\label{punt}
$[CA\vee CA',X]^{\{ u,u'\}(i\vee i')}\cong
[CA,X]^{u(i)}\times [CA',X]^{u'(i')}$
\end{remark}

\section{Generalized Homotopy Groups}
In this section the homotopy groupoid relative to a cofibration
$i:B\cof CA$ of an object X, ${\bf H_i(X)}$, is built in order to
define the first homotopy group, $\pi_1^i(X,h)$, relative to a
cofibration $i$ based on a morphism $h:CA\to X$. Then higher
homotopy groups are defined as first homotopy groups relative to
iterated cofibrations:
$\pi_n^i(X,h)=\pi_1^{i_{n-1}}(X,h\rho^{n-1})$.

Finally, main properties about the functorial character of
these groups, and their relation with coproducts and contractible
objects or cofibrations are studied.

The following commutative square is fundamental to obtain the
groupoid ${\bf H_i(X)}$:
$$\xymatrix{
  {\Sigma^i} \ar[rr]^{\rho(Ci)\cup 1} \ar@{ >->}[d]_{i_1} & & {P\{i,i\}} \ar@{ >->}[d]^{\kappa}  \\
  {C^2A} \ar[rr]_{\mu} & & {CP\{ i,i\}}}$$

$\mu$ is an extension of the morphism $\kappa\rho (Ci)\cup \kappa$
relative to the cofibration $i_1$.

Given $f_0,f_1\in Hom(CA,X)$, $H_i(f_0,f_1)$ will denote the
homotopy bracket $[C^2A,X]^{\{f_0\rho Ci,f_1\}(i_1)}$. If $F,G\in
Hom(C^2A, X)^{\{f_0\rho Ci,f_1\}(i_1)}$, then $\overline{F}$ and
$F*G$ will denote the morphisms $\{ F,f_0\rho\}\mu$ and
$\{\overline{F},G\}\mu$, respectively.

\begin{lem}\label{mu}
$\mu^*:[P\{ Ci,Ci\} ,X]^{\{ f_0,f_1\} (\kappa)}\to H_i(f_0,f_1)$
is a bijection.
\end{lem}
\begin{proof}

$(\mu^{*})^{-1}:H_i(f_0,f_1)\to [P\{ Ci,Ci\} ,X]^{\{ f_0,f_1\}
(\kappa)}$ is defined by $(\mu^{*})^{-1}([F])=[\{f_{0}\rho,F\}]$.
\begin{enumerate}

\item[-] {\it $\mu^{*}$ is well defined.} If $F:F_{0}\simeq F_{1}$
rel. $\kappa$ then $F\nu :F_{0}\mu \simeq F_{1}\mu$ rel. $i_{1}$,
where $\nu$ is an extension of the morphism $\{ (C\kappa )\mu \rho
(Ci_{1}),\kappa\mu\}$ relative to $i_{2}$.

\item[-] {\it $(\mu^{*})^{-1}$ is well defined.}
If $F:F_{0}\simeq F_{1}$ rel. $i_{1}$, then $\{ f_{0}\rho^2,F\}:
\{ f_{0}\rho,F_{0}\} \simeq\{ f_{0}\rho,F_{1}\}$ rel. $\kappa$.

\item[-] {\it $(\mu^{*})^{-1}\mu^{*}=1.$} $\{ FC\rho,\{ F\rho,G\rho\} C\mu\} :
\{ F,G\}\simeq \{f_{0}\rho,\{ F,G\}\mu \}$ rel. $\kappa$.

\item[-] {\it $\mu^{*}(\mu^{*})^{-1}=1.$} $\{ f_{0}\rho^2,F\rho\} C\mu :
F\simeq \{f_{0}\rho,F\}\mu$ rel. $i_{1}$.
\end{enumerate}
\end{proof}
\begin{thm}

${\bf H_i(X)}$ is a groupoid, with objects $Hom(CA,X)$; morphisms
from $f_0$ to $f_1$, $H_i(f_0,f_1)$; identities $1_f=[f\rho]$;
inverse morphisms $[F]^{-1}=[\overline{F}]$; and composite
morphisms $[F].[G] = [F*G]$.
\end{thm}
\begin{proof}
{\it Remark 1:} If $F:F_0\simeq F_1$ rel. $i_1$ and $G:G_0\simeq
G_1$ rel. $i_1$ then $\{ F,G\} :\{F_0,G_0\}\simeq \{ F_1,G_1\}$
rel. $\kappa$.

\begin{enumerate}
\item[-] {\it Inverse morphisms are well defined:} If $[F_0]=[F_1]
\in H_i(f_0,f_1)$, by {\it Remark 1} $\{ F_0,f_0\rho\}\simeq \{
F_1,f_0\rho\}$ rel. $\kappa$. By Lemma \ref{mu}
$[\overline{F_0}]=[\overline{F_1}]$ in $H_i(f_1,f_0)$.

\item[-] {\it Composite morphisms are well defined:} If $[F_0]=[F_1]
\in H_i(f_0,f_1)$ then $[\overline{F_0}]=[\overline{F_1}]$.
Moreover, if $[G_0]=[G_1] \in H_i(f_1,f_2)$, by {\it Remark 1} $\{
\overline{F_0}, G_0\}\simeq \{\overline{F_1},G_1 \}$ rel.
$\kappa$. By Lemma \ref{mu} $[F_0*G_0]=[F_1*G_1]\in H_i(f_0,f_2)$.
\end{enumerate}

{\it Remark 2:} By lemma \ref{mu} and the definition of inverse
morphism, $[f\rho]^{-1}=[f\rho]$ for every $f:CA\to X$.

\begin{enumerate}
\item[-] {\it Left homotopy identity property:} If $[F]\in
H_i(f_0,f_1)$, by {\it Remarks 1, 2} and Lemma \ref{mu}
$[f_0\rho*F]=[F]$.

\item[-] {\it Right homotopy identity property:} If $[F]\in
H_i(f_0,f_1)$, then $\{\{ F\rho,f_0\rho^2\}C\mu ,FC\rho\}: \{
f_0\rho,F\}\simeq \{\{ F,f_0\rho\}\mu ,f_1\rho\}$ rel. $\kappa$.
By Lemma \ref{mu} $[F] = [F*f_1\rho]$.
\end{enumerate}

{\it Remark 3:} $[\overline{\overline{F}}] = [F]$, for every
$[F]\in H_i(f_0,f_1)$, since $\overline{\overline{F}}=F*f_1\rho$.

\begin{enumerate}
\item[-] {\it Right homotopy inverse property:} If $[F]\in
H_i(f_0,f_1)$, then $\{ \overline{F}C\rho,\overline{F}C\rho\}:
\{\overline{F} ,\overline{F}\}\simeq \{f_0\rho,f_0\rho\}$ rel.
$\kappa$. By {\it Remark 2} and Lemma \ref{mu} $[F*\overline{F}] =
[f_0\rho]$.

\item[-] {\it Left homotopy inverse property:} Since the composition
of morphisms is well defined, by {\it Remark 3} and the right
homotopy inverse property  $[\overline{F}*F] =
[\overline{F}].[\overline{\overline{F}}] = [f_1\rho]$.
\end{enumerate}

{\it Remark 4:} If $[F]\in H_i(f_0,f_1)$ and $[G]\in H_i(f_1,f_2)$
then $\{\{\overline{F}\rho,G\rho\}C\mu , \overline{F}C\rho\}
:\{G,\overline{F}\}\simeq \{\{\overline{F},G\}\mu ,f_0\rho\}$ rel.
$\kappa$. By {\it Remarks 1, 3} and Lemma \ref{mu}
$[\overline{G}*\overline{F}] = [\overline{F*G}]$.

{\it Remark 5:} If $[F]\in H_i(f_0,f_1)$, $[G]\in H_i(f_1,f_2)$
and $[H]\in H_i(f_1,f_3)$ then $\{\{H\rho,\overline{F}\rho\}C\mu ,
\{ H\rho,G\rho\}C\mu\} :\{\overline{F},G\}\simeq
\{\{H,\overline{F}\}\mu ,\{ H,G\}\mu\}$ rel. $\kappa$. By {\it
Remarks 1, 3, 4} and Lemma \ref{mu} $[F*G] =
[(F*H)*(\overline{H}*G)]$.

\begin{enumerate}
\item[-] {\it Homotopy associative property:} If $[F]\in
H_i(f_0,f_1)$ , $[G]\in H_i(f_1,f_2)$ and $[H]\in H_i(f_1,f_3)$,
by the left homotopy identity property and {\it Remark 5} $H =
[f_0\rho*H] = [(f_0\rho*\overline{G})*(G*H)] =
[\overline{G}*(G*H)]$. So $[(F*G)*H] =
[(F*G)*(\overline{G}*(G*H))] = [F*(G*H)]$ since the composition of
morphisms is well defined and by {\it Remark 5}.

\end{enumerate}

\end{proof}

%\begin{eqnarray*}
%H & \stackrel{\mbox{\footnotesize
%    (Left homotopy identity property)}}{=} & [f_0p*H] =  \\
%  & \stackrel{\mbox{\it \footnotesize
%    (Remark 5)}}{=}                        & [(f_0p*\overline{G})*(G*H)] \\
%  & \stackrel{\mbox{\footnotesize
%    (Left homotopy identity property)}}{=} & [\overline{G}*(G*H)]
%\end{eqnarray*}

\begin{dfn}
The n-th homotopy group relative to a cofibration $i:B\cof CA$ of an
object $X$ based on a morphism $h:CA\to X$ is
$$\pi_n^i(X,h)= H_{i_{n-1}}(h\rho^{n-1},h\rho^{n-1}),\hspace{6pt}n\in{\Bbb N}$$

Given a cofibration $i:B\cof A$ and a morphism $h:CA\to X$, since
$\pi_n^i(X,h) = \pi_{n-s}^{i_s}(X,h\rho^s)$ the (n+1)-th homotopy
group relative to a cofibration $i:B\cof A$ of an object $X$ based
on a morphism $h:CA\to X$ is
$$\pi_{n+1}^i(X,h)=\pi_n^{i_1}(X,h)$$
\end{dfn}

Observe that $\pi_n^i(X,h)=[C^nA,X]^{h\rho^ni_n(i_n)}$.

By Theorem \ref{cont}, if $X$ or $i$ are contractible then $\pi_n^i(X,h) =
\{ 0\}$ for all $n$.

The compatibility of the homotopy relation with the composition of
morphisms (Remark \ref{hom}) gives functorial character to the
homotopy groups:

\begin{prop}\label{pifiz}
If $f:X\to Y$ is a morphism then $f_*:\pi_n^i(X,h)\to
\pi_n^i(Y,fh)$ is a homomorphism of groups.
\end{prop}
\begin{proof}
It is clearly seen that $f_*([F].[G])=f_*([F*G])=[(fF)*(fG)]$.

\end{proof}

\begin{prop}\label{pifder}
If $fi=gj$ is a is a commutative square relating cofibrations $i$
and $j$ then $(C^n f)^*:\pi_n^j(X,h)\to \pi_n^i(X,h(Cf))$ is a
homomorphism of groups.
\end{prop}
\begin{proof}
$(C^n f)^*([F].[G])=(C^n f)^*([F*G])=[(F*G)C^n f]=[(FC^n f)*(GC^n
f)]$ since $\mu_{j_n}(C^nf)i_n = (C^nf\cup C^nf)\mu_{i_n}i_n$, and
by Theorem \ref{cont} $\mu_{j_n}(C^nf)\simeq (C^nf\cup
C^nf)\mu_{i_n}$ rel. $i_n$.

\end{proof}

\begin{cor}
If $fi=gj$ is a pushout diagram then $(C^n f)^*$ is an isomorphism
of groups.
\end{cor}

In pointed categories with a natural cone, the relation between
coproducts and products in homotopy groups is a consequence of Remark
\ref{punt}:

\begin{prop}
$\pi_n^{i\vee i'}(X,\{ h,h'\} ) \cong \pi_n^{i}(X,h)\vee
\pi_n^{i'}(X,h')$.
\end{prop}
\begin{proof}
Observe that $P\{ \kappa,i\vee i'\} = P\{ \kappa\vee \kappa,i\vee
i'\} = P\{ \kappa,i\}\vee P\{ \kappa,i'\}$, where $\overline{i\vee
i'} = \overline{i}\vee \overline{i'}$ and $\overline{\kappa} =
\overline{\kappa\vee \kappa} = \overline{\kappa}\vee
\overline{\kappa}$. Hence $(i\vee i')_n = i_n\vee i'_n$, and so
that $\mu_{i_n}\vee \mu_{i'_n}$ is an extension of the type
$\mu_{(i\vee i')_n}$. Therefore $[\{ F,F'\} ].[\{ G,G'\} ] = [\{
F,F'\} *\{ G,G'\} ] = [\{ F*G,F'*G'\} ].$

\end{proof}

Classical homotopy groups in pointed categories with a natural cone
are generalized homotopy groups in the following sense:

$\pi_n^A(X)\cong \pi_n^{*_A}(X,0)$ and $\pi_n^i(X)\cong \pi_n^i(X,0)$.

In pointed categories $\pi_n^{*_A}(X,h)$ will be also denoted by
$\pi_n^A(X,h)$.

\section{Generalized Exact Homotopy Sequences}

In \cite{D-} a pair $(X,Y)$ is defined as a cofibration $f:Y\cof X$. A
morphism from $(X,Y)$ to $(X',Y')$ is a pair of morphisms $(g,h)$
such that $gf=f'h$. Cofibrations of pairs are the morphisms
$(u,v):(X,Y)\to (X',Y')$ with $v$ and $\{ f',u\} :P\{ v,f\}\to X'$
cofibrations. In this way the category of pairs {\bf cof C} of a
category with a natural cone {\bf C} is also a $C$-category. Therefore,
concepts and results obtained in the previous section are also
available in {\bf cof C}.

On the other hand, homotopy groupoids of {\bf cof C} are related with
respective ones of {\bf C} in the following sense:
\begin{enumerate}
\item[a)] If $(f_0,g_0)\simeq (f_1,g_1)$ rel. $(u,v)$ then $f_0\simeq
          f_1$ rel. $u$ and $g_0\simeq g_1$ rel. $v$.
\item[b)] $[(F_0,G_0)].[(F_1,G_1)] = [(F_0,G_0)*(F_1,G_1)] =
          [(F_0*F_1,G_0*G_1)]$.
\end{enumerate}

In this way every generalized homotopy group in {\bf C} can be also
seen as a generalized homotopy group in {\bf cof C}:

\begin{prop}\label{theta}
There is an isomorphism of groups
$$\theta_n:\pi_{n+2}^i(X,h)\to \pi_{n+1}^{(i_1,1)}((X,Y),(fh\rho ,h))$$
for all pair $(X,Y)$.
\end{prop}
\begin{proof}
If $[(F,G)]\in \pi_{n+1}^{(i_1,1)}((X,Y),(fh\rho ,h))$ then
$(F,G)(i_{n+2},1_{n+2})= (Fi_{n+2},G) \linebreak =(\{
fh\rho^{n+1}Ci_{n+1},fh\rho^n\}, h\rho^{n})$, so $G=h\rho^n$.
Therefore $F \leftrightarrow [(F,h\rho^n)]$ is an isomorphism of
groups.

\end{proof}

\begin{dfn}
The $(n+2)$-th homotopy group relative to the cofibration $i:B\cof A$
of the pair $(X,Y)$ based on the morphism $h:CA\to Y$ is
$$\pi_{n+2}^i((X,Y),h)=\pi_{n+1}^{(Ci,i)}((X,Y),(fh\rho ,h)), \hspace{7pt}
(n\in {\Bbb N})$$
\end{dfn}

Observe that $(Ci,i):(CB,B)\cof (CA,A)$ is a cofibration since $\{
Ci,\kappa\} = i_1$.

Next the exact homotopy sequence relative to a cofibration $i:B\cof
A$ associated to a pair $(X,Y)$ based on a morphism $h:CA\to Y$ is
given:

\begin{thm}\label{sucex}
The following sequence of groups is exact: $$
.....\to\pi_{3}^i(Y,h)\stackrel{f_*}{\rightarrow}\pi_{3}^{i}(X,fh)
\stackrel{j_1}{\rightarrow}\pi_3^i((X,Y),h)\stackrel{\delta_1}{\rightarrow}
\pi_{2}^i(Y,h)\stackrel{f_*}{\rightarrow}\pi_{2}^i(X,fh)$$ where
$f_*$ is the homomorphism of groups used in Proposition
\ref{pifiz}; $j_n=(1,1)^*\theta_n$, with $\theta_n$ the isomorphism of
groups given in Proposition \ref{theta} and $(1,1)^*$ is defined
using the proposition \ref{pifder} and the commutative square
$(1,1)(Ci,i) = (i_1,1)(\overline{i},i)$, where $\overline{i}$ is
the induced cofibration in the relative cone $\Sigma^i$;
$\delta_n$ is defined by $\delta_n([(F,G)])=[G]$.
\end{thm}
\begin{proof}
Clearly $f_*$, $j_n$ and $\delta_n$ are homomorphisms of groups by
their definition.

\begin{enumerate}
\item[-] $\delta_nj_n([F]) = \delta_n([(F,h\rho^{n})]) = [h\rho^{n}]$.
\item[-] $f_*\delta_{n}([(F,G)]) = f_*([G]) = [fG] = [fh\rho^n]$ since
         $HC^{n+2}\kappa:fh\rho^{n}\simeq fG$ rel. $i_{n+1}$, where $H$ is an
         extension of the morphism $\{ fh\rho^{n+2}(C^2i)_{n+1},F\}$
         relative to the cofibration $(Ci)_{n+2}$.
\item[-] $j_nf_*([F]) = j_n([fF]) = [(fF,h\rho^n)] = [(fh\rho^{n+1},h\rho^n)]$
         since:

         If $n$ is odd then $(HC^{\frac{n+3}{2}}\kappa,F):
         (fh\rho^{n+1},h\rho^n)\simeq (fF,h\rho^n)$ rel. $(Ci,i)_{n+1}$, where
         $H$ is an extension of the morphism $\{ fh\rho^{n+3}C^{n+4}i,
         fF C^{\frac{n+1}{2}}\rho,
         \linebreak fF C^{\frac{n-1}{2}}\rho,fF C^{\frac{n-3}{2}}\rho,
         ...,fF\rho,fF C^{n+1}\rho,fF C^n \rho,...,
         fF C^{\frac{n+1}{2}}\rho\}$ relative to the cofibration
         $(Ci_{\frac{n+3}{2}})_{\frac{n+3}{2}}$.

         If $n$ is even then $(HC^{\frac{n+4}{2}}\kappa,F):
         (fF,h\rho^n)\simeq (fh\rho^{n+1},h\rho^n)$ rel. $(Ci,i)_{n+1}$, where
         $H$ is an extension of the morphism $\{ fh\rho^{n+3}C^{n+4}i,
         fF C^{\frac{n+2}{2}}\rho,
         \linebreak fF C^{\frac{n}{2}}\rho,fF C^{\frac{n-2}{2}}\rho,
         ...,fF C\rho,fF C^{n+1}\rho,fF C^n \rho,...,
         fF C^{\frac{n+2}{2}}\rho,fh\rho^{n+2}\}$
         \linebreak relative to the cofibration
         $(Ci_{\frac{n+2}{2}})_{\frac{n+4}{2}}$.
\item[-] If $\delta ([(F,G)])=[G]=[h\rho^n]$ then there is $H:h\rho^n\simeq
         G$ rel. $i_{n+1}$. $j_n([KC\kappa ])=[(KC\kappa,h\rho^n)]=[(F,G)]$ since
         $(K,H):(KC\kappa,h\rho^n)\simeq (F,G)$ rel. $(Ci,i)_{n+1}$, where
         $K$ is an extension of the morphism $\{ fh\rho^{n+2}C^{n+3}i,
         \linebreak fH,fh\rho^{n+1},...,fh\rho^{n+1},F\}$ relative to the
         cofibration $(Ci_{n+1})_1$.
\item[-] If $f_*([F])=[fF]=[fh\rho^n]$ then there is $G:fh\rho^n\simeq fF$
         rel. $i_{n+1}$. Let $H$ be an extension of the morphism $\{
         fh\rho^{n+2}C^{n+3}i,G,fh\rho^{n+1},...,fh\rho^{n+1}\}$ relative to
         the cofibration $Ci_{n+2}$, then $\delta ([(H\kappa,F)]) = [F]$.
\item[-] If $j_n([F])=[(F,h\rho^n)]=[(fh\rho^{n+1},h\rho^n)]$:

         {\it When $n$ is even}, taking $(H,G):(F,h\rho^n)\simeq
         (fh\rho^{n+1},h\rho^n)$ rel. $(Ci,i)_{n+1}$, $(H_1C^{n+3}\kappa )i_{n+3}
         = \{ fh\rho^{n+2}C^{n+2}i_1,fG,fh\rho^{n+1},...,fh\rho^{n+1},F\}$,
         where $H_1$ is an extension of the morphism $\{ fh\rho^{n+3}C^{n+4}i,
         F\rho,H,fh\rho^{n+2},...,
         \linebreak fh\rho^{n+2},FC^n\rho,FC^{n+1}\rho\}$ relative to
         the cofibration $(Ci)_{n+3}$.

         $(H_2C^{n+2}\kappa )i_{n+3} = \{ fh\rho^{n+2}C^{n+1}i_2,fG,fh\rho^{n+1},...,
         fh\rho^{n+1},F,fh\rho^{n+1}\}$, whe\-re $H_2$ is an extension of the
         morphism $\{ fh\rho^{n+3}C^{n+3}i_1,F\rho,H_1,
         \linebreak fh\rho^{n+2},..., fh\rho^{n+2}, FC^n\rho,FC^{n-1}\rho\}$
         relative to the cofibration $(Ci_1)_{n+2}$.

         $(H_3C^{n+1}\kappa )i_{n+3} = \{ fh\rho^{n+2}C^{n}i_3,fG,fh\rho^{n+1},...,
         fh\rho^{n+1},F\}$, where $H_3$ is an extension of the
         morphism $\{ fh\rho^{n+3}C^{n+2}i_2,F\rho,H_2,fh\rho^{n+2},...,
         \linebreak fh\rho^{n+2},
         FC^{n-2}\rho,FC^{n-1}\rho\}$ relative to the cofibration $(Ci_2)_{n+1}$.

         $(H_4C^{n}\kappa )i_{n+3} = \{ fh\rho^{n+2}C^{n-1}i_4,fG,fh\rho^{n+1},...,
         fh\rho^{n+1},F,fh\rho^{n+1}\}$, whe\-re $H_4$ is an extension of the
         morphism $\{ fh\rho^{n+3}C^{n+1}i_3,F\rho,H_3,fh\rho^{n+2},
         \linebreak ...,fh\rho^{n+2},FC^{n-2}\rho,FC^{n-3}\rho\}$ relative to
         the cofibration $(Ci_3)_{n}$.

         This process can be iterated to obtain a homotopy $H_{n+1}C^3\kappa:
         fG\simeq F$ rel. $i_{n+2}$.

         {\it When $n$ is odd}, the same process for a homotopy $(H,G):
         (fh\rho^{n+1},h\rho^n)
         \linebreak \simeq (F,h\rho^n)$ rel. $(Ci,i)_{n+1}$ gives
         $H_{n+1}C^3\kappa: fG\simeq F$ rel. $i_{n+2}$.
\end{enumerate}
\end{proof}

In pointed categories with a natural cone \cite{D-} classical exact
homotopy sequences associated to a pair are obtained as generalized
exact homotopy sequences taking the morphism $h=0$.

\section{Spherical Homotopy Groups}
Every category with a natural cylinder and product is a category
with a natural cone and the same cofibrations \cite{D-}. Therefore
the category of topological spaces with the topological cone has a
structure of natural cone. This topological cone is obtained as
the pushout of the inclusion in the base of the cylinder with the
projection on a point. In this way the cone of the empty
topological space is a point. So the category of pointed
topological spaces is a category under the cone of an object, that
obtains homotopy groups using spheres.

The development described above can be generalized for any category
with a natural cone:

The cone $C\nabla$ of any object $\nabla$ behave like a point. Fixed an
object $\nabla$, ${\bf C}^*$ is the full subcategory of ${\bf C}^{C\nabla}$
whose objects, $(X,x)$, are cofibrations $x:C\nabla\cof X$.

The cone functor $C_*:{\bf C}^*\to {\bf C}^*$ is defined by
$C_*(X,x) = (C_*X,\overline{Cx})$, where $C_*X=P\{ \rho,Cx\}$ and
given $f:(X,x)\to (Y,y)$, $C_*f=1\cup Cf:P\{ \rho,Cx\}\to P\{
\rho,Cy\}$.

Every pointed object $(X,x)$ has associated the trivial pushout
diagram of the object $P\{ 1,x\}$. However other pushout diagrams
can be also associated to the object $(X,x)$: Any cofibration
pushout diagram with induced cofibration $x$ will be considered a
pushout diagram associated to $(X,x)$.

\begin{thm}\label{copunt}
$C^n_*(X,x)=(P\{ \rho^n,C^nx\} ,\stackrel{(n}{\overline{Cx}})$,
where $\stackrel{(0}{\overline{Cx}}=x$ and
$\stackrel{(n}{\overline{Cx}}=\overline{C\stackrel{(n-1}{\overline{Cx}}}$
is obtained by iterated use of the functor $C_*$.
\end{thm}
\begin{proof}
It is enough to observe that the following square diagram is a
pushout:
$$\xymatrix{
  {C^{n+1}\nabla} \ar[rr]^{\rho^n} \ar@{ >->}[d]_{C^nx} & & {C\nabla} \ar@{ >->}[d]^{\stackrel{(n}{\overline{Cx}}}  \\
  {C^nX} \ar[rr]_{\overline{\rho}C\overline{\rho}...C^{n-1}\overline{\rho}} & & {C^n_*X}}$$

\begin{enumerate}
\item[-] {\it Commutativity:}
      \begin{eqnarray*}
      \overline{\rho}C\overline{\rho}...C^{n-1}\overline{\rho}C^nx
        & = & \overline{\rho}C\overline{\rho}...C^{n-1}
              \stackrel{(1}{\overline{Cx}}C^{n-1}p =       \\
        & = & \overline{\rho}C\overline{\rho}...C^{n-2}
              \stackrel{(2}{\overline{Cx}}(C^{n-2}\rho )(C^{n-1}\rho ) =...  \\
        & = & \stackrel{(n}{\overline{Cx}}\rho (C\rho )....(C^{n-2}\rho )
              (C^{n-1}\rho ) = \stackrel{(n}{\overline{Cx}}\rho^n
      \end{eqnarray*}
% $\overline{p}C\overline{p}...C^{n-1}
%         \overline{p}C^nx = \overline{p}C\overline{p}...C^{n-1}
%         \stackrel{(1}{\overline{Cx}}C^{n-1}p = \overline{p}C\overline{p}
%         ...C^{n-2}\stackrel{(2}{\overline{Cx}}(C^{n-2}p)(C^{n-1}p) =....=
%         \stackrel{(n}{\overline{Cx}}p(Cp)....(C^{n-2}p)(C^{n-1}p) =
%         \stackrel{(n}{\overline{Cx}}p^n$.
\item[-] {\it Pushout property:}

         Given $F:C\nabla\to Y$ and $G:C^n X\to Y$ with $F\rho^n=
         GC^nx$, then
         $F\rho^{n-1}C^{n-1}p = GC^{n-1}(Cx)$, and there is $\{
         F\rho^{n-1},G\}:C^{n-1}(C_* X)\to Y$. $\{ F\rho^{n-1},G\}C^{n-2}
         C\overline{Cx} = \{ F\rho^{n-1},G\}C^{n-1}\overline{Cx}=
         F\rho^{n-2}C^{n-2}\rho$, and there is $\{ F\rho^{n-2},F\rho^{n-1},G\}:
         C^{n-2}C_*^2 X\to Y$. This process can be iterated to obtain
         the unique $\{ F,F\rho,...,F\rho^{n-1},G\}=\{ F,G\}:C_*^n X=P\{
         \rho^n,C^n x\}\to Y$.
\end{enumerate}
\end{proof}

Natural transformations $\kappa_*:1\to C_*$ and $\rho_*:C^2_*\to
C_*$ are defined by ${\kappa_{*}}_{(X,x)} = 1\cup
\kappa_X:(X,x)=P\{ 1,x\}\to C_*(X,x)=P\{ \rho,Cx\}$ and
${\rho_{*}}_{(X,x)} = 1\cup \rho_X:C_*^2(X,x) = P\{ \rho^2,C^2
x\}\to C_*(X,x)=P\{ \rho,Cx\}$

\begin{prop}
Given $f:(X,x)\to (Y,y)$, $C_*^n f = 1\cup C^nf:C_*^n(X,x) = P\{
\rho^n,C^nx\}\to C_*^n(Y,y) = P\{ \rho^n,C^ny\}$;
${\kappa_*}_{C_*^n(X,x)} = 1\cup \kappa_{C^n X} : C_*^n(X,x)=P\{
\rho^n,C^nx\} \to C_*^{n+1}(X,x) =  P\{ \rho^{n+1},C^{n+1}x\}$ and
${\rho_*}_{C_*^n(X,x)} = 1\cup \rho_{C^n X} : C_*^{n+2}(X,x)=P\{
\rho^{n+2},C^{n+2}x\} \to C_*^{n+1}(X,x) =  P\{
\rho^{n+1},C^{n+1}x\}$.
\end{prop}
\begin{proof}
$1\cup \kappa_{C^n X} =
\{\stackrel{(n+1}{\overline{Cx}},\overline{\rho}
C\overline{\rho}...C^n\overline{\rho}kC^n_X\}$. So that $(1\cup
\kappa_{C^n X}) \stackrel{(n}{\overline{Cx}} =
\stackrel{(n+1}{\overline{Cx}}$ and \linebreak = $(1\cup
\kappa_{C^n X})1_{C^n_*X}= \overline{\rho}\kappa_{C^n_* X}$ since
$(\overline{\rho}\kappa_{C^n_*X}) \stackrel{(n}{\overline{Cx}} =
\overline{\rho}C\stackrel{(n}{\overline{Cx}} \kappa_{C\nabla} =
\stackrel{(n+1}{\overline{Cx}}\rho_{C\nabla}\kappa_{C\nabla} =
\linebreak \stackrel{(n+1}{\overline{Cx}}$ and
$(\overline{\rho}\kappa_{C^n_*X})
\overline{\rho}C\overline{\rho}...C^{n-1}\overline{\rho} =
\overline{\rho}C\overline{\rho}...C^{n}\overline{\rho}\kappa_{C^nX}$.

Hence the morphisms $1\cup \kappa_{C^n X} : C_*^n(X,x)=P\{
\rho^n,C^nx\} \to C_*^{n+1}(X,x) =  P\{ \rho^{n+1},C^{n+1}x\}$ and
$1\cup \kappa_{C_*^n X} : C_*^n(X,x)=P\{
1,\stackrel{(n}{\overline{Cx}}\} \longrightarrow C_*^{n+1}(X,x) =
\linebreak P\{ \rho,C\stackrel{(n}{\overline{Cx}}\}$ agree.

The other equalities are trivial by definition.

\end{proof}

A morphism $i:(B,b)\to (A,a)$ is called a {\it pointed cofibration} when
$i:B\cof A$ is a cofibration in {\bf C}.

\begin{thm}\label{punto}
${\bf C}^*$, with the cone $C_*$, natural transformations
$\kappa_*$, $\rho_*$ and the pointed cofibrations, is a pointed
category with a natural cone.
\end{thm}
\begin{proof}
Observe that $(C\nabla ,1)$ is an initial object with initial
morphisms, $x:(C\nabla ,1)\cof (X,x)$, pointed cofibrations; so
every object is cofibrant. Clearly $C_*(C\nabla ,1) = (P\{
\rho,1\},\overline{1}) = (C\nabla ,1)$.

\begin{enumerate}
\item[] {\it Cone axiom:} It is a simple verification.
\item[] {\it Pushout axiom:} Given a pointed cofibration
        $i:(B,b)\cof (A,a)$ and a morphism $f:(B,b)\to (X,x)$, then
        $P\{ f,i\} = (P\{ f,i\} ,x\cup a)$ with induced morphisms
        $\overline{f}:(A,a)\to (P\{ f,i\} ,x\cup a)$ and
        $\overline{i}:(X,x)\to (P\{ f,i\} ,x\cup a)$, where $x\cup
        a:C\nabla = P\{ 1,1\} \cof P\{ f,i\}$ is a cofibration by
        Theorem \ref{cubo} since $\{ a,i\}=i:P\{ 1,b\}=B\cof A$. Note
        that $\overline{i}=1\cup \overline{i}:P\{ 1,x\}\cof P\{
        1,x\cup a\}$ is a cofibration.

        $P\{ C_*f,C_*i\} = P\{ 1\cup Cf,1\cup Ci\} = P\{ \rho\cup \rho,
        Cx\cup Ca\} = C_*P\{ f,i\}$, with induced morphisms
        $\overline{C_*f} = 1\cup C\overline{f}$ and $\overline{C_*i}
        = 1\cup C\overline{i}$.
\item[] {\it Cofibration axiom:} Clearly $1_{(X,x)}$ and the
        composition of pointed cofibrations are pointed cofibrations.
        $\kappa_{(X,x)}=1\cup \kappa_X$ is a cofibration by Theorem \ref{cubo}
        since $x_1$ so is. Given a pointed cofibration $i:(B,b)\cof
        (A,a)$, by NEP in {\bf C} there is $r:CA\to CB$ such that
        $r(Ci)=1$. Then $1\cup r:C_*(A,a)=P\{ \rho,Ca\}\to C_*(B,b)=P\{
        \rho,Cb\}$ is a retraction for $C_*i$.
\item[] {\it Relative cone axiom.} Given a pointed cofibration $i:(B,b)\cof
        (A,a)$, $\{ C_*i,\kappa_*\}:P\{\kappa_*,i\}\to C_*(A,a)$ will be
        denoted by $i_{1*}:(\Sigma_*^i,\overline{Cb}\cup a)\to C_*(A,a)$.

        $\Sigma_*^i=P\{ 1\cup \kappa,1\cup i\}=P\{ \rho\cup 1,Cb\cup a\}$
        with induced morphisms $\overline{\kappa_*}=1\cup \overline{\kappa}$ and
        $\overline{i}=1\cup \widetilde{i}$, where $\widetilde{i}$ is
        the induced cofibration by $i$ in the pushout of $P\{ \kappa,i\}$.

        $(1\cup i_1)(1\cup \overline{\kappa}) = 1\cup \kappa = \kappa_*$ and
        $(1\cup i_1)(1\cup \widetilde{i}) = 1\cup Ci = C_*i$. So
        $i_{1*} = 1\cup i_1$, that is a cofibration by Theorem
        \ref{cubo} since $\{ i_1,Ca\} =i_1:P\{ Cb\cup a,1\} =\Sigma^i
        \cof CA$.
\end{enumerate}
\end{proof}

The concepts developed in pointed categories with a natural cone
\cite{D-} are also available in ${\bf C}^*$. They can be related with
the respective ones of the original category {\bf C}.

\begin{prop}
Given $f:(X,x)\to (Y,y)$, $f\simeq 0$ in ${\bf C}^*$ if and only if
$f\simeq 0$ in {\bf C}.
\end{prop}
\begin{proof}
If $F:f\simeq 0$ in ${\bf C}^*$ then $F\overline{\rho}:f\simeq 0$
in {\bf C}, where $\overline{\rho}$ is the induced morphism in the
pushout of $C_*X$.

If $F:f\simeq 0$ in {\bf C} then $\{ FCxC\rho,F\}:\{
y\rho,f\}\simeq 0$, and by NEP there is an extension $H$ of the
morphism $\{ y\rho,f\}$ relative to the cofibration $x_1$. $\{
y,H\}:f\simeq 0$ in ${\bf C}^*$.

\end{proof}

\begin{cor}
$(X,x)\simeq 0$ if and only if $X\simeq 0$.
\end{cor}

Next pointed homotopy groups in ${\bf C}^*$ will be stated as
generalized homotopy groups of {\bf C}. A study about pushout
diagrams associated to pointed objects is necessary for it.

\begin{prop}\label{copun}
If $P\{ s,j\}$ is a pushout diagram associated to $(X,x)$ then
$C_*^n(X,x) = P\{ \rho^nC^ns,C^nj\}$.
\end{prop}
\begin{proof}
It is enough to observe the following composition of pushouts:

\begin{picture}(300,75)(0,0)
\put(60,50){\makebox(0,0){$C^n S$}}
\put(170,50){\makebox(0,0){$C^{n+1}\nabla$}}
\put(280,50){\makebox(0,0){$C\nabla$}}
\put(60,10){\makebox(0,0){$C^n T$}}
\put(170,10){\makebox(0,0){$C^n$X}}
\put(280,10){\makebox(0,0){$\widehat{C}^n$X}}

\put(80,50){\vector(1,0){70}}
         \put(115,58){\makebox(0,0){$C^ns$}}
\put(190,50){\vector(1,0){70}}
         \put(225,58){\makebox(0,0){$\rho^n$}}
\put(80,10){\vector(1,0){70}}
         \put(115,18){\makebox(0,0){$C^n\overline{s}$}}
\put(190,10){\vector(1,0){70}}
         \put(225,18){\makebox(0,0){$\overline{\rho}C\overline{\rho}...
                                    C^{n-1}\overline{\rho}$}}
\put(60,40){\cab{13}{}}
         \put(35,25){$C^nj$}
\put(170,40){\cab{13}{}}
         \put(145,25){$C^nx$}
\put(280,40){\cab{13}{}}
         \put(285,23){$\stackrel{(n}{\overline{Cx}}$}
\end{picture}
\end{proof}

\begin{cor}\label{morpun}
Given $g=1\cup f:(X,x)=P\{ s,j\}\to (X',x')=P\{ s',j'\}$:
\begin{enumerate}
\item[a)] $C_*^ng=1\cup C^nf:C_*^n(X,x)=P\{ \rho^nC^ns,C^nj\}\to
          C_*^n(X',x')= P\{ \rho^nC^ns',C^nj'\}$.
\item[b)] ${\kappa_*}_{C_*^n(X,x)}=1\cup \kappa_{C^nT}:C_*^n(X,x)=P\{ \rho^nC^ns,C^nj\}\to
          C_*^{n+1}(X,x)=P\{ \rho^{n+1}C^{n+1}s,C^{n+1}j\}$.
\item[c)] ${\rho_*}_{C_*^n(X,x)}=1\cup \rho_{C^nT}:C_*^{n+2}(X,x)=
          P\{ \rho^{n+2}C^{n+2}s,C^{n+2}j\}\to C_*^{n+1}(X,x)=
          P\{ \rho^{n+1}C^{n+1}s,C^{n+1}j\}$.
\end{enumerate}
\end{cor}
\begin{proof}
\begin{enumerate}
\item[(a)] $1\cup C^nf=\{ \stackrel{(n}{\overline{Cx'}},\overline{\rho}
           C\overline{\rho}...C^{n-1}\overline{\rho}C^n\overline{s'}
           C^nf\}$. So that:

           $(1\cup C^nf) \stackrel{(n}{\overline{Cx}} \hspace{4pt} =
           \hspace{4pt} \stackrel{(n}{\overline{Cx'}}$, and
           \begin{eqnarray*}
           (1\cup C^nf)(\overline{\rho}C\overline{\rho}...C^{n-1}\overline{\rho})& = &
                     \{ \stackrel{(n}{\overline{Cx'}}\rho^n,
                     \overline{\rho}C\overline{\rho}...C^{n-1}\overline{\rho}C^n
                     \overline{s'}C^nf\} =                              \\
               & = & \{ \overline{\rho}C\overline{\rho}...C^{n-1}
           \overline{\rho}C^nx',\overline{\rho}C\overline{\rho}...C^{n-1}
           \overline{\rho}C^n \overline{s'}C^nf\} =                        \\
               & = & \overline{\rho}C\overline{\rho}...C^{n-1}
           \overline{\rho}\{ C^nx',C^n \overline{s'}C^nf\} =                \\
               & = & \overline{\rho}C\overline{\rho}...C^{n-1}
           \overline{\rho}(1\cup C^nf) =                                    \\
               & = & (\overline{\rho}C\overline{\rho}...C^{n-1}
                     \overline{\rho})C^ng
\end{eqnarray*}
\item[(b)] $1\cup \kappa_{C^nT}=\{ \stackrel{(n+1}{\overline{Cx}},\overline{\rho}
           C\overline{\rho}...C^{n}\overline{\rho}C^{n+1}\overline{s}
           \kappa_{C^nT}\}$. So that:

           $(1\cup \kappa_{C^nT})\stackrel{(n}{\overline{Cx}} \hspace{4pt}
           = \hspace{4pt} \stackrel{(n+1}{\overline{Cx}}$, and
           \begin{eqnarray*}
           (1\cup \kappa_{C^nT})(\overline{\rho}C\overline{\rho}...C^{n-1}\overline{\rho})
              & = &  \{ \stackrel{(n+1}{\overline{Cx}}\rho^n,
                     \overline{\rho}C\overline{\rho}...C^{n}\overline{\rho}C^{n+1}
                     \overline{s}\kappa_{C^nT}\} =                              \\
              & = &  \{ \stackrel{(n+1}{\overline{Cx}}\rho^{n+1}
                     \kappa_{C^n\nabla},\overline{\rho}C\overline{\rho}...C^{n}
                     \overline{\rho}C^{n+1}\overline{s}\kappa_{C^nT}\} =           \\
              & = & \overline{\rho}C\overline{\rho}...C^{n}
                    \overline{\rho}\{ C^{n+1}x\kappa_{C^{n+1}\nabla},C^{n+1}
                    \overline{s}\kappa_{C^nT}\} =                               \\
              & = & \overline{\rho}C\overline{\rho}...C^{n}\overline{\rho}
                    (\kappa_{C^{n+1}\nabla}\cup \kappa_{C^nT}) =                     \\
              & = & (\overline{\rho}C\overline{\rho}...C^{n}\overline{\rho})
                    \kappa_{C^n X}
\end{eqnarray*}
\item[(c)] $1\cup \rho_{C^nT}=\{ \stackrel{(n+1}{\overline{Cx}},\overline{\rho}
           C\overline{\rho}...C^{n}\overline{\rho}C^{n+1}\overline{s}
           \rho_{C^nT}\}$. So that:

           $(1\cup \rho_{C^nT})\stackrel{(n+2}{\overline{Cx}} \hspace{4pt} =
           \hspace{4pt}\stackrel{(n+1}{\overline{Cx}}$, and
           \begin{eqnarray*}
           (1\cup \rho_{C^nT})(\overline{\rho}C\overline{\rho}...C^{n+1}\overline{\rho})
              & = &  \{ \stackrel{(n+1}{\overline{Cx}}\rho^{n+2},
                     \overline{\rho}C\overline{\rho}...C^{n}\overline{\rho}C^{n+1}
                     \overline{s}\rho_{C^nT}\} =                             \\
%              & = & \{ \overline{p}C\overline{p}...C^{n}
%                    \overline{p} C^{n+1}xp_{C^n\nabla},\overline{p}C
%                    \overline{p}...C^{n} \overline{p}C^{n+1}\overline{s}
%                    p_{C^nT}\} =                                          \\
                    {}\hspace{15pt}{} & = & \overline{\rho}C\overline{\rho}...
                    C^{n}\overline{\rho}(\rho_{C^{n+1}\nabla}\cup \rho_{C^nT})=     \\
              & = & (\overline{\rho}C\overline{\rho}...C^{n}\overline{\rho})\rho_{C^n X}
\end{eqnarray*}
\end{enumerate}
\end{proof}

\begin{dfn}
A pointed cofibration $u:(B,b)\cof (A,a)$ is said to have
associated a pushout cofibration when $u=1\cup i:(B,b)=P\{
s,j\}\cof (A,a)=P\{ s,j'\}$, where $i:T\cof T'$ is a cofibration
verifying $j'=ij$.
\end{dfn}

Observe that every pointed cofibration $i:(B,b)\cof (A,a)$ has an
associated pushout cofibration $1\cup i:(B,b)=P\{ 1,b\}\cof
(A,a)=P\{ 1,a\}$.

On the other hand, every object $X$ can be considered as an
iterated pushout diagram: $X=P\{ 1_X,1_X\} = P\{ 1_X^1,1_X\} = P\{
1_X^2,1_X\} = ...  = P\{ 1_X^n,1_X\}$, where $1_X^0 = 1_X$ and
$1_X^r = 1_X:X\cof P\{ 1_X^{r-1},1_X\}$. In this way it is
possible to define curly braces or unions with domain the object
$X$.

\begin{thm}
Given a pushout cofibration $1\cup i:(B,b)=P\{ s,j\}\cof (A,a)=P\{
s,j'\}$ associated to a pointed cofibration $u$, the following
square diagram is a pushout:

\begin{picture}(300,110)(0,0)
\put(40,70){\makebox(0,0){$\Sigma^{(1_S)_n} = C^{n+1}S$}}
\put(340,70){\makebox(0,0){$C\nabla$}}
\put(40,20){\makebox(0,0){$\Sigma^{i_n}$}}
\put(340,20){\makebox(0,0){$\Sigma_*^{u_{n*}}$}}

\put(90,70){\vector(1,0){220}}
  \put(190,80){\makebox(0,0){$\rho^{n+1}C^{n+1}s$}}
\put(60,20){\vector(1,0){260}}
  \put(190,10){\makebox(0,0){\footnotesize $\overline{\rho}C\overline{\rho}...
              C^n\overline{\rho}C^{n+1}\overline{s}\cup
              \overline{\rho}C\overline{\rho}...C^{n-1}\overline{\rho}C^n\overline{s}
              \cup \stackrel{(n+1}{........}\cup \overline{\rho}C\overline{\rho}
              ...C^{n-1}\overline{\rho}C^n\overline{s}$}}
\put(40,60){\cab{23}{}}
  \put(105,45){\makebox(0,0){\footnotesize $C^{n+1}j\cup C^nj'\cup
             \stackrel{(n+1}{........}\cup C^nj'$}}
\put(340,60){\cab{23}{}}
  \put(285,50){\makebox(0,0){\footnotesize $\stackrel{(n+1}{\overline{Cb}}
               \cup \stackrel{(n}{\overline{Ca}}\cup
               \stackrel{(n+1}{........}
               \cup \stackrel{(n}{\overline{Ca}}$}}
\end{picture}

$(n\in {\Bbb N}^*$ and $C^{-1}\overline{\rho} C^0\overline{s} =
\overline{s}$)

such that $u_{n+1*} = 1\cup i_{n+1}:\Sigma_*^{u_{n*}} = P\{
\rho^{n+1}C^{n+1}s, C^{n+1}j\cup C^nj'\cup
\stackrel{(n+1}{........}\cup C^nj'\}\to C_*^{n+1}(A,a) = P\{
\rho^{n+1}C^{n+1}s,C^{n+1}j'\}$.
\end{thm}
\begin{proof}
\begin{enumerate}
\item[] {\it Commutativity:} It is enough to observe that
        $\rho^{n+1}C^{n+1}s\cup \rho^nC^ns \cup \stackrel{(n+1}{........}\cup
        \rho^nC^ns=\rho^{n+1}C^{n+1}s:\Sigma^{(1_S)_n} = C^{n+1}S\to P\{
        1_{C\nabla}^n,1_{C\nabla}\} = C\nabla$.
\item[] {\it Pushout property:} Given $F=\{ F,...,F\}:C\nabla = P\{
        1^n_{C\nabla},1_{C\nabla}\} \to X$ and $\{G_{n+1},...,G_0\}
        :\Sigma^{i_n}\to X$ such that $F\rho^{n+1}C^{n+1}s =
        \{F\rho^{n+1}C^{n+1}s,
        \linebreak F\rho^{n}C^{n}s,\stackrel{(n+1}{........},
        F\rho^{n}C^{n}s\} = \{ G_{n+1}C^{n+1}j,G_nC^nj',...,G_0C^nj'\}$.
        Then $G_{n+1}C^{n+1}j=F\rho^{n+1}C^{n+1}s$  and $G_{i}C^{n}j' =
        F\rho^{n}C^{n}s$, $0\leq i\leq n$. By Proposition \ref{copun}
        there are $\{ F,G_{n+1}\}:C_*^{n+1}B\to X$ and $\{ F,G_i\}:
        C_*^nA\to X$. By Corollary \ref{morpun} $C_*^r\kappa_* = 1\cup
        C^r\kappa$ and $C_*^nu=1\cup C^ni$, therefore the existence of the
        morphism $\{G_{n+1},...,G_0\}$ allows to define the unique
        morphism $\{\{ F,G_{n+1}\} ,...,\{ F,G_0\}\} = \{ F,G_{n+1},
        ...,G_0\}:\Sigma_*^{u_{n*}}=P\{ \rho^{n+1}C^{n+1} s,C^{n+1}j
        \cup C^nj'\cup \stackrel{(n+1}{........}\cup C^nj'\}\to {\rm
        X}$.
\item[]
        \vspace{-20pt}
\begin{eqnarray*}
u_{n+1*} & = & \{ C_*^{n+1}(1\cup i),C_*^n \kappa_*,...,\kappa_*\} =     \\
 & = & \{ 1\cup C^{n+1}i,1\cup C^n \kappa,...,1\cup \kappa\} =           \\
 & = & \{\{ \stackrel{(n+1}{\overline{Ca}},\overline{\rho}C\overline{\rho}...
       C^n\overline{\rho}C^{n+1}\overline{s}C^{n+1}i\},
       \{\stackrel{(n+1}{\overline{Ca}},\overline{\rho}C\overline{\rho}...
       C^{n}\overline{\rho}C^{n+1}\overline{s}C^n\kappa\},             \\
 &  & ....,\{\stackrel{(n+1}{\overline{Ca}},\overline{\rho}C\overline{\rho}...
       C^{n}\overline{\rho}C^{n+1}\overline{s} \kappa\}\} \approx      \\
 & \approx & \{ \stackrel{(n+1}{\overline{Ca}},\overline{\rho}
             C\overline{\rho}...             C^n\overline{\rho}C^{n+1}
             \overline{s}\{ C^{n+1}i,C^n\kappa,...,\kappa\}\} =          \\
 & = & \{ \stackrel{(n+1}{\overline{Ca}},\overline{\rho}C\overline{\rho}...
       C^n\overline{\rho}C^{n+1}\overline{s}i_{n+1}\} =1\cup i_{n+1}
\end{eqnarray*}
\end{enumerate}

\end{proof}

Based objects are necessary to obtain pointed homotopy groups.
Given a pushout associated to a pointed object $(A,a)$ with
pointed base morphism $\alpha$, then $\alpha\overline{s}:T\to
C\nabla$ verifies $\alpha\overline{s}j=s$. Conversely, every
morphism $\alpha :T\to C\nabla$ such that $\alpha j=s$ gives a
pointed base morphism $1\cup \alpha :(A,a,\alpha)=P\{ s,j\}\to
C\nabla = P\{ 1,1\}$. $(A,a,\alpha )=P\{ s,j\}$ will denote a base
pointed object, where $\alpha :T\to C\nabla$.

Observe that 0-morphism is  $\{ x,x\alpha\} :(A,a,\alpha )=P\{
s,j\}\to (X,x)$. Moreover a pushout cofibration $u=1\cup
i:(B,b,\beta )=P\{ s,j\} \cof (A,a,\alpha )=P\{ s,j'\}$ is based
if and only if $\alpha i=\beta$.

\begin{thm}
Given a based pushout cofibration $u=1\cup i:(B,b,\beta ) = P\{
\rho (Cs),j\} \cof C_*(A,a,\alpha )=P\{ \rho (Cs),Cj'\}$, then
$$[(C_*A,\overline{Ca}),(X,x)]^{(u)}\cong [CT',X]^{x\beta (i)}$$
\end{thm}
\begin{proof}
Given $g:CT'\to X$, $gCj'=gij=x\beta j=x\rho (Cs)$. Hence there is
$\{ x,g\}:C_*(A,a) = P\{ \rho (Cs),Cj'\}\to (X,x)$. The bijection
$\theta: [CT',X]^{x\beta (i)}\to
[(C_*A,\overline{Ca}),(X,x)]^{(u)}$ is defined by $\theta
([g])=[\{ x,g\} ]$.

\begin{enumerate}
\item[] {\it $\theta$ is well defined} since $G:g_0\simeq g_1$ rel.
        $i$ implies that $\{ x,G\}:C_*^2(A,a)=P\{ \rho^2C^2s,C^2j'\}\to
        (X,x)$ verifies $\{ x,G\}:\{x ,g_0\}\simeq \{ x,g_1\}$ rel. $u$.
\item[] Clearly {\it $\theta$ is suprajective.}
\item[] {\it $\theta$ is injective} since if $G:\{ x,g_0\}\simeq \{
x,g_1\}$ rel. $u$ then $G\overline{\rho^2C^2s}:g_0\simeq g_1$ rel.
$i$.
\end{enumerate}
\end{proof}

\begin{cor}\label{grupar}
Given a based pushout cofibration $u=1\cup i:(B,b,\beta )=P\{
s,j\}\cof (A,a,\alpha)=P\{ s,j'\}$, then
$$\pi_{n*}^u((X,x))\cong \pi_n^i(X,x\rho (C\alpha ))$$
\end{cor}
\begin{proof}
\begin{eqnarray*}
\pi_{n*}^u((X,x)) & =& [C^n_*(A,a),(X,x)]^{(u_{n*})} \cong     \\
  & \cong & [C^nT',X]^{x\{\rho^nC^n\beta,\rho^{n-1}C^{n-1}\alpha,...,
           \rho^{n-1}C^{n-1}\alpha\} (i_n)} =                   \\
  & = & [C^nT',X]^{x\rho^nC^n\alpha i_n(i_n)} =         \\
  & = & \pi_n^i(X,x\rho (C\alpha ))
\end{eqnarray*}

Given an extension $\mu$ of the morphism $\kappa\rho
(Ci_{n-1})\cup \kappa$ relative to the cofibration $i_n$, then
$\mu C^nj' = (C^nj'\cup C^nj')\kappa \rho$ and $\rho^nC^ns =
(\rho^nC^ns)\kappa \rho$. Hence there is $1\cup\mu:C^n_*(A,a)=P\{
\rho^nC^ns,C^nj'\} \longrightarrow (C_*P\{
u_{n-1*},u_{n-1*}\},\overline{u_{n-1*}}
\stackrel{(n-1}{\overline{Ca}}) = P\{ \rho^nC^ns,C^nj'\cup
C^nj'\}$ verifying
\begin{eqnarray*}
(1\cup\mu )u_{n*} & \approx & (1\cup\mu )(1\cup i_{n}) =      \\
 & = & 1\cup(\kappa\rho (Ci_{n-1})\cup \kappa )\approx                       \\
 & \approx & ((1\cup \kappa )(1\cup \rho )(1\cup Ci_{n-1}))\cup (1\cup \kappa ) =  \\
 & = & \kappa_*\rho_*Cu_{n-1*}\cup \kappa_*
\end{eqnarray*}
Therefore the bijection above is an isomorphism of groups.

\end{proof}

\begin{cor}\label{obj}
Given a based pointed object $(A,a,\alpha )$:
$$\pi_{n*}^{A}((X,x))\cong \pi_n^a(X,x\rho (C\alpha) )$$
\end{cor}

\begin{cor}
The pointed exact sequence relative to the based pushout
cofibration $u=1\cup i$ of the pointed pair $((X,x),(Y,y))$ is
isomorphic to the exact sequence relative to the cofibration $i$
of the pair $(X,Y)$ based on the morphism $x\rho (C\alpha )$
\end{cor}

The isomorphism of groups given in Corollary \ref{grupar} let one to
extend the definition of pointed homotopy groups to pointed objects
$(X,x)$ where $x$ be not a cofibration:
$$\pi_{n*}^u((X,x)) = \pi_n^i(X,x\rho (C\alpha) )$$

In this way, although ${\bf C}^{C\nabla}$ has not, in general, a
natural cone in the sense described in this paper, it is possible to
obtain homotopy groups of its objects, as it occurs with topological
spaces and pointed topological spaces.

Next, taking as example topological spheres, spherical objects are defined.

\begin{dfn}
The $0$-sphere of ${\bf C}^{C\nabla}$ is $S^0 = P\{
\kappa_{\nabla} ,\kappa_{\nabla}\}$.
\end{dfn}

Observe that $(S^0,\overline{\kappa},\{ 1,1\})$ is a based pointed
object.

\begin{dfn}
The $n$-sphere of ${\bf C}^{C\nabla}$ is $S^n =
S^n_*(S^0,\overline{\kappa})$. \cite{D-}
%({\bf Habr  que poner la def. de suspensi¢n en los preliminares})
\end{dfn}

\begin{remark}
Note that, for $n\geq 2$
\begin{eqnarray*}
\pi_{n*}^{S^0}((X,x)) & = & [S^n,(X,x)] =
                         (\mbox{Corollary \ref{obj}}) =  \\
  & = & \pi_n^{\overline{\kappa}}(X,\{ x\rho,x\rho\}) \cong
        (\mbox{pushout of definition of $S^0$}) \cong   \\
  & \cong & \pi_n^{\kappa_{\nabla}}(X,x)
\end{eqnarray*}

Hence one can to extend the definition of $\pi_{n*}^{S^0}((X,x))$ for
$n=1$ by:
$$\pi_{1*}^{S^0}((X,x)) = \pi_1^{\kappa_{\nabla}}(X,x)$$
\end{remark}

\begin{dfn}
$\pi_{n*}^{S^0}((X,x))$ are denominated Spherical Homotopy Groups of
the pointed object $(X,x)$.
\end{dfn}

\begin{remark}
If {\bf C} is a pointed category with a natural cone, then
$\pi_{n*}^{S^0}((X,x)) = \pi_{n+1}^{\nabla}(X,x)$.

If {\bf C} has initial object $\emptyset$ and
$C\emptyset\neq\emptyset$, then $\pi_{n+1}^{\emptyset}(X,x)=
\pi_{n*}^{S^0}((X,x))$ are denominated Standard Spherical Homotopy
Groups of the pointed object $(X,x)$. $\pi_{1*}^{S^0}((X,x)) =
\pi_{2}^{\emptyset}(X,x)$ is also denominated Fundamental Group.
\end{remark}

In the category of topological spaces, the standard spherical
homotopy groups of a pointed topological space are the classical
homotopy groups.

\end{document}